\title{On various (strong) rainbow connection numbers of graphs\footnote{Supported by NSFC No.11371205 and 11531011}}

\author{Lin Chen$^1$, Xueliang Li$^{1,3}$, Henry Liu$^2$, Jinfeng Liu$^1$\\
\\
\normalsize $^1$Center for Combinatorics and LPMC\\
\normalsize Nankai University, Tianjin 300071, China\\
\normalsize chenlin1120120012@126.com, lxl@nankai.edu.cn, ljinfeng709@163.com\\
\\
\normalsize $^2$School of Mathematics and Statistics\\
\normalsize Central South University, Changsha 410083, China\\
\normalsize henry-liu@csu.edu.cn\\
\\
\normalsize $^3$School of Mathematics and Statistics\\
\normalsize Qinghai Normal University, Xining, Qinghai 810008, China
}
\date{18 April 2017}

\documentclass[11pt]{article}
\usepackage{amsmath,amssymb,amscd,amsthm}
\usepackage{xcolor}
\usepackage{multirow}
\newdimen\unit\newdimen\psep\newcount\nd\newcount\ndx\newbox\dotb\newbox\ptbox
\newdimen\dx\newdimen\dy\newdimen\dxx\newdimen\dyy\newdimen\hgt
\newdimen\xoff\newdimen\yoff
\newcommand\clap[1]{\hbox to 0pt{\hss{#1}\hss}}
\newcommand\vdisk[1]{{\font\dotf=cmr10 scaled #1\dotf.}}
\newcommand\varline[2]{\setbox\dotb\hbox{\vdisk{#1}}\xoff=-.5\wd\dotb
\wd\dotb=0pt\yoff=-.5\ht\dotb\psep=#2\ht\dotb}
\newcommand\varpt[1]{\setbox\ptbox\clap{\vdisk{#1}}\setbox\ptbox
\hbox{\raise-.5\ht\ptbox\box\ptbox}}
\newcommand\cpt{\copy\ptbox}
\newcommand\point[3]{\rlap{\kern#1\unit\raise#2\unit\hbox{#3}}}
\newcommand\setnd[4]{\dx=#3\unit\advance\dx-#1\unit\divide\dx by\psep
\dy=#4\unit\advance\dy-#2\unit\divide\dy by\psep \multiply\dx
by\dx\multiply\dy by\dy\advance\dx\dy\nd=1\advance\dx-1sp
\loop\ifnum\dx>0\advance\dx-\nd sp\advance\nd1\advance\dx-\nd
sp\repeat}
\newcommand\dl[4]{{\setnd{#1}{#2}{#3}{#4}\dline{#1}{#2}{#3}{#4}\nd}}
\newcommand\dline[5]{{\nd=#5\hgt=#2\unit\dx=#3\unit\advance\dx-#1\unit
\divide\dx by\nd\dy=#4\unit\advance\dy-#2\unit\divide\dy by\nd
\advance\hgt\yoff\rlap{\kern#1\unit\kern\xoff\loop\ifnum\nd>1\advance\nd-1
\advance\hgt\dy\kern\dx\raise\hgt\copy\dotb\repeat}}}

\newcommand\qellip[4]{{\setnd{0}{0}{#3}{#4}\dx=\unit\dy=0pt\raise\yoff\rlap{%
\kern#1\unit\kern\xoff\raise#2\unit\hbox{\loop\ifnum\dx>0\rlap{\kern#3\dx
\raise#4\dy\copy\dotb}\hgt=\dx\divide\hgt
by\nd\advance\dy\hgt\hgt=\dy \divide\hgt
by\nd\advance\dx-\hgt\repeat\rlap{\raise#4\dy\copy\dotb}}}}}
\newcommand\bez[6]{{\setnd{#1}{#2}{#3}{#4}\ndx=\nd\setnd{#3}{#4}{#5}{#6}
\ifnum\ndx>\nd\nd=\ndx\fi\dx=#3\unit\advance\dx-#1\unit\dy=#4\unit
\advance\dy-#2\unit\dxx=#5\unit\advance\dxx-#1\unit\dyy=#6\unit\advance
\dyy-#2\unit\advance\dxx-2\dx\advance\dyy-2\dy\divide\dxx
by\nd\divide\dyy
by\nd\advance\dx.25\dxx\advance\dy.25\dyy\divide\dx
by\nd\divide\dy by\nd \multiply\nd
by2\dx=100\dx\dy=100\dy\dxx=100\dxx\dyy=100\dyy\divide\dxx by\nd
\divide\dyy
by\nd\hgt=#2\unit\raise\yoff\rlap{\kern#1\unit\kern\xoff
\raise\hgt\copy\dotb\loop\ifnum\nd>0\advance\nd-1\advance\hgt0.01\dy
\kern0.01\dx\raise\hgt\copy\dotb\advance\dx\dxx\advance\dy\dyy\repeat}}}
\newcommand\ptu[3]{\point{#1}{#2}{\cpt\raise1ex\clap{$\scriptstyle{#3}$}}}
\newcommand\ptd[3]{\point{#1}{#2}{\cpt\raise-1.8ex\clap{$\scriptstyle{#3}$}}}
\newcommand\ptr[3]{\point{#1}{#2}{\cpt\raise-.4ex\rlap{$\ \scriptstyle{#3}$}}}
\newcommand\ptl[3]{\point{#1}{#2}{\cpt\raise-.4ex\llap{$\scriptstyle{#3}\ $}}}
\newcommand\ptlu[3]{\point{#1}{#2}{\raise.8ex\clap{$\scriptstyle{#3}$}}}
\newcommand\ptld[3]{\point{#1}{#2}{\raise-1.6ex\clap{$\scriptstyle{#3}$}}}
\newcommand\ptlr[3]{\point{#1}{#2}{\raise-.4ex\rlap{$\,\scriptstyle{#3}$}}}
\newcommand\ptll[3]{\point{#1}{#2}{\raise-.4ex\llap{$\scriptstyle{#3}\,$}}}
\newcommand\pt[2]{\point{#1}{#2}{\cpt}}

\newcommand\thnline{\varline{400}{.6}}
\newcommand\dotline{\varline{1000}{4}}
\varpt{2500}\thnline\unit=2em

\newtheorem{theorem}                   {Theorem}
\newtheorem{thm}             [theorem] {Theorem}

\newtheorem{lem}             [theorem] {Lemma}

\newtheorem{prop}             [theorem] {Proposition}

\newtheorem{prob}            [theorem] {Problem}

\newtheorem*{clm1'}{Claim 1$'$}
\newtheorem*{clm2'}{Claim 2$'$}
\newtheorem*{clm3'}{Claim 3$'$}
\newtheorem*{clm4'}{Claim 4$'$}

\def\diam{\textup{diam}}

\begin{document}
\maketitle

\begin{abstract}
An edge-coloured path is \emph{rainbow} if all the edges have distinct colours. For a connected graph $G$, the \emph{rainbow connection number} $rc(G)$ is the minimum number of colours in an edge-colouring of $G$ such that, any two vertices are connected by a rainbow path. Similarly, the \emph{strong rainbow connection number} $src(G)$ is the minimum number of colours in an edge-colouring of $G$ such that, any two vertices are connected by a rainbow geodesic (i.e., a path of shortest length). These two concepts of connectivity in graphs were introduced by Chartrand et al.~in 2008. Subsequently, vertex-coloured versions of both parameters, $rvc(G)$ and $srvc(G)$, and a total-coloured version of the rainbow connection number, $trc(G)$, were introduced. In this paper we introduce the strong total rainbow connection number $strc(G)$, which is the version of the strong rainbow connection number using total-colourings. Among our results, we will determine the strong total rainbow connection numbers of some special graphs. We will also compare the six parameters, by considering how close and how far apart they can be from one another. In particular, we will characterise all pairs of positive integers $a$ and $b$ such that, there exists a graph $G$ with $trc(G)=a$ and $strc(G)=b$, and similarly for the functions $rvc$ and $srvc$.
\\[2mm]
[2mm] \textbf{AMS Subject Classification (2010):} 05C12, 05C15, 05C40
\end{abstract}

\section{Introduction}
In this paper, all graphs under consideration are finite and simple. For notation and terminology not defined here, we refer to \cite{B98}.

In 2008, Chartrand et al.~\cite{CJMZ} introduced the concept of rainbow connection of graphs. An edge-coloured path is \emph{rainbow} if all of its edges have distinct colours. Let $G$ be a non-trivial connected graph. An edge-colouring of $G$ is \emph{rainbow connected} if any two vertices of $G$ are connected by a rainbow path. The minimum number of colours in a rainbow connected edge-colouring of $G$ is the \emph{rainbow connection number} of $G$, denoted by $rc(G)$. The topic of rainbow connection is fairly interesting and numerous relevant papers have been published. In addition, the concept of strong rainbow connection was introduced by the same authors. For two vertices $u$ and $v$ of $G$, a \emph{$u-v$ geodesic} is a $u-v$ path of length $d(u,v)$, where $d(u,v)$ is the distance between $u$ and $v$. An edge-colouring of $G$ is \emph{strongly rainbow connected} if for any two vertices $u$ and $v$ of $G$, there is a rainbow $u-v$ geodesic. The minimum number of colours in a strongly rainbow connected edge-colouring of $G$ is the \emph{strong rainbow connection number} of $G$, denoted by $src(G)$. The investigation of $src(G)$ is slightly more challenging than that of $rc(G)$, and fewer papers have been obtained on it. For details, see \cite{CJMZ,CL,LLL,LS2}.

As a natural counterpart of rainbow connection, Krivelevich and Yuster \cite{KY} proposed the concept of rainbow vertex-connection. A vertex-coloured path is \emph{vertex-rainbow} if all of its internal vertices have distinct colours. A vertex-colouring of $G$ is \emph{rainbow vertex-connected} if any two vertices of $G$ are connected by a vertex-rainbow path. The minimum number of colours in a rainbow vertex-connected vertex-colouring of $G$ is the \emph{rainbow vertex-connection number} of $G$, denoted by $rvc(G)$. Corresponding to the strong rainbow connection, Li et al.~\cite{LMS} introduced the notion of strong rainbow vertex-connection. A vertex-colouring of $G$ is \emph{strongly rainbow
vertex-connected} if for any two vertices $u$ and $v$ of $G$, there is a vertex-rainbow $u-v$ geodesic. The minimum number of colours in a strongly rainbow vertex-connected vertex-colouring of $G$ is the \emph{strong rainbow vertex-connection number} of $G$, denoted by $srvc(G)$. For more results on rainbow vertex-connection, we refer to \cite{LS,LiuMSrvc}.

It was also shown that computing the rainbow connection number and rainbow vertex-connection number of an arbitrary graph is NP-hard \cite{CLRTY, CFMY,CLL, CLS, HLS,LLS2015}. For more results on the rainbow connection subject, we refer to the survey \cite{LSS} and the book \cite{LS1}.

Subsequently, Liu et al.~\cite{LiuMS} proposed the concept of total rainbow connection. A total-coloured path is \emph{total-rainbow} if its edges and internal vertices have distinct colours. A total-colouring of $G$ is \emph{total rainbow connected} if any two vertices of $G$ are connected by a total-rainbow path. The minimum number of colours in a total rainbow connected total-colouring of $G$ is the \emph{total rainbow connection number} of $G$, denoted by $trc(G)$. For more results on the total rainbow connection number, see \cite{JLZ,Ma}. Inspired by the concept of strong rainbow (vertex-)connection, a natural idea is to introduce the \emph{strong total rainbow connection number}. A total-colouring of $G$ is \emph{strongly total rainbow connected} if for any two vertices $u$ and $v$ of $G$, there is a total-rainbow $u-v$ geodesic. The minimum number of colours in a strongly total rainbow connected total-colouring of $G$ is the \emph{strong total rainbow connection number} of $G$, denoted by $strc(G)$.

Very recently, Dorbec et al.~\cite{PIE2014} initiated the study of rainbow connection in digraphs. Subsequently, versions of the other five parameters for digraphs were considered. For more details, see \cite{JJ2015,JJJ2015,AJ2016,HMS2014,LLLS,LLMS}.

This paper will be organised as follows. In Section \ref{remarkssect}, we will present results for all six rainbow connection parameters for general graphs. In Section \ref{specificsect}, we determine the strong total rainbow connection number of some specific graphs, including cycles, wheels and complete bipartite and multipartite graphs. Finally in Section \ref{comparesect}, we will compare the six parameters, by considering how close and how far apart they can be from one another. In particular, we will characterise all pairs of integers $a$ and $b$ such that, there exists a connected graph $G$ with $rvc(G)=a$ and $srvc(G)=b$, and similarly for the functions $trc$ and $strc$.

We mention a few more words on terminology and notation. For a graph $G$, its vertex and edge sets are denoted by $V(G)$ and $E(G)$, and its diameter is denoted by $\diam(G)$. Let $K_n$ and $C_n$ denote the complete graph and cycle of order $n$ (where $n\ge 3$ for $C_n$), and $K_{m,n}$ denote the complete bipartite graph with class sizes $m$ and $n$. For two graphs $G$ and $H$, and a vertex $u\in V(G)$, we define $G_{u\to H}$ to be the graph obtained by replacing $u$ with $H$, and replacing the edges of $G$ at $u$ with all edges between $H$ and the neighbours of $u$ in $G$. We say that $G_{u\to H}$ is \emph{obtained from $G$ by expanding $u$ into $H$}. Note that the graph obtained from $G$ by expanding every vertex into $H$ is also known as the \emph{lexicographic product $G\circ H$}.

\section{Remarks and results for general graphs}\label{remarkssect}

In this section, we present some results about the six rainbow connection parameters $rc(G)$, $src(G),rvc(G),srvc(G),trc(G)$ and $strc(G)$, for general graphs $G$. Let $G$ be a non-trivial connected graph with $m$ edges and $n$ vertices, where $q$ vertices are non-pendent (i.e., with degree at least $2$). We have the following inequalities.
\begin{eqnarray}
&\diam(G) \le rc(G)\le src(G)\le m,\label{ineq1}\\[0.5ex]
&\diam(G)-1 \le rvc(G)\le srvc(G)\le \min(n-2,q),\label{ineq2}\\[0.5ex]
&2\,\diam(G)-1 \le trc(G)\le strc(G)\le srvc(G)+m\le\min(m+n-2,m+q),\label{ineq3}
\end{eqnarray}
To see the last inequality of (\ref{ineq2}), the inequality $srvc(G)\le n-2$ is a result of Li et al.~\cite{LMS}. We also have $srvc(G)\le q$, since any vertex-colouring of $G$ where all $q$ non-pendent vertices are given distinct colours, is strongly rainbow vertex-connected. To see the third inequality of (\ref{ineq3}), we may take a strongly rainbow vertex-connected colouring of $G$ with $srvc(G)$ colours, and then colour the edges with $m$ further distinct colours. This gives a strongly total rainbow connected colouring of $G$ with $srvc(G)+m$ colours. The last inequality of (\ref{ineq3}) then follows from (\ref{ineq2}). All remaining inequalities are trivial.

Also, the following upper bound is obvious.
\[
strc(G)\leq src(G)+q.
\]
Indeed, a strongly total rainbow connected colouring of $G$ can be obtained from a strongly rainbow colouring with $src(G)$ colours, and then colouring the non-pendent vertices of $G$ with $q$ further distinct colours. Similarly, for graphs with diameter $2$, we have the following proposition which will be very helpful for later.

\begin{prop}\label{diam=2}
Let $G$ be a graph with diameter $2$. Then $strc(G)\leq src(G)+1$.
\end{prop}

\begin{proof}
By definition, we may give $G$ a strongly rainbow connected colouring, using $src(G)$ colours. Since $\diam(G)=2$, any two non-adjacent vertices $x,y\in V(G)$ are connected by a rainbow $x-y$ geodesic of length $2$. Now, colour all vertices of $G$ with a new colour. Then clearly, the resulting total-colouring uses $src(G)+1$ colours, and is a strongly total rainbow connected colouring. Thus, $strc(G)\leq src(G)+1$.
\end{proof}

For the functions $rc(G)$ and $trc(G)$, we have the following upper bounds which are better than those of (\ref{ineq1}) and (\ref{ineq3}).
\[
rc(G)\le n-1,\quad\textup{and}\quad trc(G)\le\min(2n-3,n-1+q).
\]
Indeed, we may take a spanning tree $T$ of $G$, which has $n-1$ edges and at most $\min(n-2,q)$ non-pendent vertices. We can assign distinct colours to all edges of $T$, and to all edges and non-pendent vertices of $T$, to obtain, respectively, the above two upper bounds.

As for alternative lower bounds instead of those involving the diameter, we note that for any total rainbow connected colouring of $G$, the colours of the bridges and cut-vertices must be pairwise distinct. Similar observations hold for rainbow connected and rainbow vertex-connected colourings, where respectively, the colours of the bridges, and the colours of the cut-vertices, must be pairwise distinct. Hence, the following result holds.

\begin{prop}\label{bridges}
Let $G$ be a connected graph. Suppose that $B$ is the set of all bridges, and $C$ is the set of all cut-vertices. Denote $b=|B|$ and $c=|C|$, respectively. Then
\begin{eqnarray}
& src(G)\ge rc(G)\ge b,\nonumber\\[0.5ex]
& srvc(G)\ge rvc(G)\ge c,\nonumber\\[0.5ex]
& strc(G)\ge trc(G)\ge b+c.\nonumber
\end{eqnarray}
\end{prop}

In the next result, we give equivalences and implications when the rainbow connection parameters are small.

\begin{thm}\label{thm1}
Let $G$ be a non-trivial connected graph.
\begin{enumerate}
\item[(a)] The following are equivalent.
\begin{enumerate}
\item[(i)] $G$ is a complete graph.
\item[(ii)] $\diam(G)=1$.
\item[(iii)] $rc(G)=1$.
\item[(iv)] $src(G)=1$.
\item[(v)] $rvc(G)=0$.
\item[(vi)] $srvc(G)=0$.
\item[(vii)] $trc(G)=1$.
\item[(viii)] $strc(G)=1$.
\end{enumerate}
\item[(b)] $strc(G)\ge trc(G)\ge 3$ if and only if $G$ is not a complete graph.
\item[(c)]
\begin{enumerate}
\item[(i)] $rc(G) = 2$ if and only if $src(G) = 2$.
\item[(ii)] $rvc(G) = 1$, if and only if $srvc(G) = 1$, if and only if $\diam(G) = 2$.
\item[(iii)] $rvc(G)=2$ if and only if $srvc(G)=2$.
\item[(iv)] $trc(G) = 3$ if and only if $strc(G) = 3$.
\item[(v)] $trc(G) = 4$ if and only if $strc(G) = 4$.
\end{enumerate}
Moreover, any of the conditions in (i) implies any of the conditions in (iv), and any of the conditions in (i), (iv) and (v) implies any of the conditions in (ii).
\end{enumerate}
\end{thm}

\begin{proof}
Although parts of this result can be found in \cite{CJMZ,LMS}, we provide a proof for the sake of completeness.\\[1ex]
\indent(a) Clearly we have (ii) $\Rightarrow$ (i) $\Rightarrow$ (iv). Using (\ref{ineq1}), we can easily obtain (iv) $\Rightarrow$ (iii) $\Rightarrow$ (ii). Similarly, using (\ref{ineq2}) and (\ref{ineq3}), we have (ii) $\Rightarrow$ (i) $\Rightarrow$ (vi) $\Rightarrow$ (v) $\Rightarrow$ (ii), and (ii) $\Rightarrow$ (i) $\Rightarrow$ (viii) $\Rightarrow$ (vii) $\Rightarrow$ (ii).\\[1ex]
\indent(b) If $G$ is not a complete graph, then $\diam(G)\ge 2$, and $strc(G)\ge trc(G)\ge 3$ follows from (\ref{ineq3}). The converse clearly holds by (a).\\[1ex]
\indent(c) We first prove (i). Suppose first that $src(G)=2$. Then by (a), we have $\diam(G)\ge 2$. By (\ref{ineq1}), we have $2\le rc(G)\le src(G) = 2$, and hence $rc(G) = 2$. Conversely, suppose that $rc(G) = 2$. Then (a) and (\ref{ineq1}) imply that $src(G)\ge 2$ and $\diam(G)=2$. Also, there exists a rainbow connected colouring for $G$, using $rc(G) = 2$ colours. In such an edge-colouring, for any $x,y \in V(G)$, either $xy \in E(G)$, or $xy \not\in E(G)$ and there is a rainbow $x-y$ path of length $2$, which is also a rainbow $x - y$ geodesic. Thus $src(G)\le 2$, and $src(G) = 2$ as required.

By similar arguments using (a), (\ref{ineq2}) and (\ref{ineq3}) we can prove (iv); that the first two conditions of (ii) are equivalent; and that the first condition of (v) implies the second.

Now we complete the proof of (ii). If $rvc(G)=1$, then we can easily use (a) and (\ref{ineq2}) to obtain $\diam(G)=2$. If $\diam(G)=2$, then (\ref{ineq2}) gives $rvc(G)\ge 1$. Clearly, the vertex-colouring of $G$ where every vertex is given the same colour is rainbow vertex-connected, and thus $rvc(G)\le 1$. Therefore (ii) holds.

Next, we prove (iii). Suppose first that $srvc(G)=2$. Then $rvc(G)\le 2$ by (\ref{ineq2}). Clearly $rvc(G)\neq 0$ by (a), and $rvc(G)\neq 1$ by (c)(ii). Thus $rvc(G)=2$. Conversely, suppose that $rvc(G)=2$. Then by (\ref{ineq2}), we have $srvc(G)\ge 2$ and $\diam(G)\le 3$. We may take a rainbow vertex-connected colouring of $G$, using at most $rvc(G)=2$ colours. Let $x,y\in V(G)$. If $d(x,y)\in\{1,2\}$, then any $x-y$ geodesic is clearly vertex-rainbow. If $d(x,y)=3$, then since any $x-y$ path of length at least $4$ cannot be vertex-rainbow, there must exist a vertex-rainbow $x-y$ path of length $3$, which is also an $x-y$ geodesic. Thus, the colouring is also strongly rainbow vertex-connected. We have $srvc(G)\le 2$, so that $srvc(G)=2$, and (iii) holds.

Next, we complete the proof of (v). Suppose that $strc(G)=4$. By (a) and (b), we have $3\le trc(G)\le strc(G)=4$. By (iv), we have $trc(G)\neq 3$, so that $trc(G)=4$. Thus (v) holds.

Finally, we prove the last part of (c). Firstly, suppose that either condition in (i) holds, so that $rc(G) = 2$. Then (a) and (b) imply $trc(G)\ge 3$. Moreover, there exists a rainbow connected edge-colouring for $G$, using $rc(G) = 2$ colours. Clearly by colouring all vertices of $G$ with a third colour, we have a total rainbow connected colouring for $G$, using $3$ colours. Thus, $trc(G)\le 3$. We have $trc(G) = 3$, and thus both conditions of (iv) hold. Secondly, suppose that any of the conditions in (i), (iv) or (v) holds. It is easy to use (a), and (\ref{ineq1}) or (\ref{ineq3}), to obtain $\diam(G)=2$. Thus, the three conditions of (ii) also hold.
\end{proof}

\noindent\textbf{Remark.} We remark that in Theorem \ref{thm1}(c), no other implication exists between the conditions of (i) to (v). Obviously, no implication exists between the conditions of (ii) and those of (iii). Thus by the last part of (c), no implication exists between the conditions of (iii) and those of (i), (iv) and (v). Similarly, no implication exists between the conditions of (iv) and those of (v), and thus no implication exists between the conditions of (i) and those of (v), since the conditions of (i) imply those of (iv). Clearly, the example of the stars $K_{1,n}$ shows that there are infinitely many graphs where the conditions of (ii) hold, but those of (i), (iv) and (v) do not hold. Indeed, for $n\ge 2$, we have $rvc(K_{1,n})=1$, while $rc(K_{1,n})=n$ and $trc(K_{1,n})=n+1$. Now, there are infinitely many graphs $G$ such that the conditions of (iv) hold, but those of (i) do not hold. For example, let $u$ be a vertex of the cycle $C_5$, and let $G$ be a graph obtained by expanding $u$ into a clique $K$. That is, $G=(C_5)_{u\to K}$. It was remarked in \cite{LiuMS} (and also easy to show) that for any such graph $G$, we have $trc(G)=rc(G) = 3$.
\\[2ex]
\indent Now, it is easy to see that if $H$ is a spanning connected subgraph of a connected graph $G$, then we have
\[
rc(G)\le rc(H),\quad rvc(G)\le rvc(H),\quad\textup{and}\quad trc(G)\le trc(H).
\]
However, the following lemma shows that the same inequalities do not hold for the strong rainbow connection parameters.

\begin{lem}\label{srcG>srcHlem}
There exist connected graphs $G$ and $H$ such that, $H$ is a spanning subgraph of $G$, and $src(G) > src(H)$. Similar statements hold for the functions $srvc$ and $strc$.
\end{lem}

\begin{proof}
We construct graphs $G_i$ and $H_i$, for $i=1,2,3$, as follows. Let $H_1$ (resp.~$H_2,H_3$) be the graph as shown in Figure 1(a) (resp.~(b), (c)) consisting of the solid edges, and $G_1$ (resp.~$G_2,G_3$) be the graph obtained by adding the dotted edge.\\[1ex]
\[ \unit = 1cm
\thnline
\dl{-3}{1}{-5}{-1}\dl{-3}{-1}{-5}{1}\dl{-7}{1}{-5}{-1}\dl{-7}{-1}{-5}{1}\dl{-6}{0}{-5}{0.33}\dl{-6}{0}{-5}{-0.33}\dl{-4}{0}{-5}{0.33}\dl{-4}{0}{-5}{-0.33}
\pt{-3}{1}\pt{-3}{-1}\pt{-4}{0}\pt{-5}{1}\pt{-5}{0.33}\pt{-5}{-0.33}\pt{-5}{-1}\pt{-6}{0}\pt{-7}{1}\pt{-7}{-1}
\dotline
\dl{-4}{0}{-6}{0}
\ptlu{-6.4}{0.45}{1}\ptld{-6.4}{-0.5}{2}\ptlu{-5.6}{0.45}{3}\ptld{-5.6}{-0.5}{4}\ptlu{-5.32}{0.15}{4}\ptld{-5.32}{-0.2}{3}\ptlu{-4.4}{0.45}{1}\ptld{-4.4}{-0.5}{2}\ptlu{-4.68}{0.15}{1}\ptld{-4.68}{-0.2}{2}\ptlu{-3.6}{0.45}{3}\ptld{-3.6}{-0.5}{4}
\point{-5.22}{-1.8}{\small (a)}
\thnline
\dl{-0.5}{1}{0.5}{1}\dl{-0.5}{-1}{0.5}{-1}\dl{-2.5}{-1}{-0.5}{1}\dl{-2.5}{1}{-0.5}{-1}\dl{2.5}{-1}{0.5}{1}\dl{2.5}{1}{0.5}{-1}\dl{-1.5}{0}{-0.5}{0.33}\dl{-1.5}{0}{-0.5}{-0.33}\dl{-0.5}{0.33}{0.5}{0.33}\dl{-0.5}{-0.33}{0.5}{-0.33}\dl{0.5}{0.33}{1.5}{0}\dl{0.5}{-0.33}{1.5}{0}
\pt{-2.5}{-1}\pt{-2.5}{1}\pt{-2}{-0.5}\pt{-2}{0.5}\pt{-1.5}{0}\pt{-0.5}{1}\pt{-0.5}{-1}\pt{-0.5}{-0.33}\pt{-0.5}{0.33}\pt{0.5}{1}\pt{0.5}{-1}\pt{0.5}{-0.33}\pt{0.5}{0.33}\pt{1.5}{0}\pt{2}{0.5}\pt{2}{-0.5}\pt{2.5}{1}\pt{2.5}{-1}
\dotline
\dl{-1.5}{0}{0.5}{-1}
\ptlu{-1.9}{0.45}{1}\ptld{-1.9}{-0.5}{2}\ptlu{1.9}{0.45}{3}\ptld{1.9}{-0.5}{4}\ptlu{-1.5}{0}{5}\ptlu{1.5}{0}{6}\ptlu{-0.5}{1}{3}\ptlu{-0.5}{0.33}{4}\ptlu{-0.5}{-0.33}{3}\ptld{-0.5}{-1.04}{4}\ptlu{0.5}{1}{1}\ptlu{0.5}{0.33}{1}\ptlu{0.5}{-0.33}{2}\ptld{0.5}{-1.04}{2}
\point{0.4}{-0.85}{\small $x$}
\point{-0.25}{-1.8}{\small (b)}
\thnline
\dl{3}{1}{5}{-1}\dl{3}{-1}{5}{1}\dl{7}{1}{5}{-1}\dl{7}{-1}{5}{1}\dl{6}{0}{5}{0.33}\dl{6}{0}{5}{-0.33}\dl{4}{0}{5}{0.33}\dl{4}{0}{5}{-0.33}
\pt{3}{1}\pt{3}{-1}\pt{4}{0}\pt{5}{1}\pt{5}{0.33}\pt{5}{-0.33}\pt{5}{-1}\pt{6}{0}\pt{7}{1}\pt{7}{-1}\pt{3.5}{0.5}\pt{3.5}{-0.5}\pt{6.5}{0.5}\pt{6.5}{-0.5}
\dotline
\dl{4}{0}{6}{0}
\ptlu{3.6}{0.45}{1}\ptld{3.6}{-0.5}{2}\ptlu{4.4}{0.45}{9}\ptld{4.4}{-0.5}{10}\ptlu{4.68}{0.15}{10}\ptld{4.68}{-0.2}{9}\ptlu{5.6}{0.45}{7}\ptld{5.6}{-0.5}{8}\ptlu{5.32}{0.15}{7}\ptld{5.32}{-0.2}{8}\ptlu{6.4}{0.45}{3}\ptld{6.4}{-0.5}{4}\ptlu{4}{0}{5}\ptlu{6}{0}{6}\ptlu{5}{1}{1}\ptlu{5}{0.33}{1}\ptld{5}{-0.35}{2}\ptld{5}{-1.04}{2}\ptld{3.62}{0.34}{7}\ptlu{3.62}{-0.41}{8}\ptld{3.12}{0.8}{11}\ptlu{3.12}{-0.82}{12}\ptld{6.38}{0.36}{9}\ptlu{6.45}{-0.41}{10}\ptld{6.88}{0.8}{13}\ptlu{6.88}{-0.82}{14}
\point{4.8}{-1.8}{\small (c)}
\ptlu{0}{-2.8}{\textup{\small Figure 1. The graphs in Lemma \ref{srcG>srcHlem}}}
\]\\[1ex]
\indent We will prove that
\begin{equation}\label{G>Heq}
src(G_1)>src(H_1),\quad srvc(G_2)>srvc(H_2),\quad\textup{and}\quad strc(G_3)>strc(H_3).
\end{equation}

Firstly, it is easy to see that the edge-colouring of $H_1$ as shown is strongly rainbow connected, and thus $src(H_1)\le 4$. In fact, we have $src(H_1)=4$, since $src(H_1)\ge\diam(H_1)=4$. Now, suppose that there exists a strongly rainbow connected colouring of $G_1$, using at most four colours. Note that the four pendent edges of $G_1$ must received distinct colours, say colours $1,2,3,4$. The dotted edge has colour $1,2,3$ or $4$, and in each case, we can easily find two vertices that are not connected by a rainbow geodesic in $G_1$. We have a contradiction, and thus $src(G_1)\ge 5>4=src(H_1)$.

We can similarly prove the remaining two inequalities of (\ref{G>Heq}). We have a strongly rainbow vertex-connected colouring of $H_2$ as shown, and since $\diam(H_2)=7$, we have $srvc(H_2)=6$. Suppose that there exists a strongly rainbow vertex-connected colouring of $G_2$, using at most six colours. Then, the six cut-vertices of $G_2$ must received distinct colours, say colours $1,2,3,4,5,6$. The vertex $x$ has colour $1,2,3,4,5$ or $6$, and in each case, we can find two vertices that are not connected by a vertex-rainbow geodesic in $G_2$. We have a contradiction, and thus $srvc(G_2)\ge 7>6=srvc(H_2)$. Likewise, we have a strongly total rainbow connected colouring of $H_3$ as shown, and thus $strc(H_3)\le 14$. Suppose that there exists a strongly total rainbow connected colouring of $G_3$, using at most $14$ colours. Then, the eight bridges and six cut-vertices of $G_3$ must received distinct colours, say colours $1,2,\dots,14$. The dotted edge has colour $1,2,\dots,13$ or $14$, and in each case, we can find two vertices that are not connected by a total-rainbow geodesic in $G_3$. Again we have a contradiction, and thus $strc(G_3)\ge 15>14\ge strc(H_3)$.
\end{proof}

Li et al.~\cite{LMS} provided a similar example of graphs $G$ and $H$ which gave $srvc(G) = 9 > 8 = srvc(H)$. However in their example, $H$ was not a spanning subgraph of $G$, although this could be easily corrected. Chartrand et al.~\cite{CJMZ} had conjectured that $src(G) \le src(H)$ whenever $G$ and $H$ are connected graphs, with $H$ a spanning subgraph of $G$. They observed that if this conjecture was true, then we have $src(G)\le n-1$ if $G$ is a connected graph of order $n$. However, Lemma \ref{srcG>srcHlem} shows that the conjecture is false. The latter claim may still be true, and we propose this as an open problem, as well as the total-coloured analogue.

\begin{prob}
Let $G$ be a connected graph of order $n$ with $q$ non-pendent vertices. Then, are the following inequalities true?
\[
src(G)\le n-1,\quad\text{and}\quad strc(G)\le\min(2n-3,n-1+q).
\]
\end{prob}

\section{Strong total rainbow connection numbers of some graphs}\label{specificsect}

In this section, we consider the strong total rainbow connection numbers of some specific graphs, namely, trees, cycles, wheels, and complete bipartite and multipartite graphs. The remaining five rainbow connection parameters for these graphs have previously been considered by various authors, and we shall recall these previous results along the way.

First, let $T$ be a tree of order $n$, with $q$ non-pendent vertices. Note that, since any two vertices of $T$ are connected by a unique path, we have $rc(T)=src(T)$, $rvc(T)=srvc(T)$, and $trc(T)=strc(T)$. From Chartrand et al.~\cite{CJMZ}, and Liu et al.~\cite{LiuMSrvc,LiuMS}, we have $rc(T)=src(T)=n-1$, $rvc(T)=q$, and $trc(T)=n-1+q$. Moreover, it is well known that if $n\ge 3$, then $1\le q\le n-2$; and that $q=1$ if and only if $T$ is a star, and $q=n-2$ if and only if $T$ is a path. Thus, we have the following result.

\begin{prop}\label{strc(T)}
Let $T$ be a tree with order $n$, and $q$ non-pendent vertices.
\begin{enumerate}
\item[(a)] $rvc(T)=srvc(T)=q$. In particular, for $n\ge 2$, $rvc(T)=srvc(T)=n-2$ if and only if $T$ is a path; and for $n\ge 3$, $rvc(T)=srvc(T)=1$ if and only if $T$ is a star.
\item[(b)] $trc(T)=strc(T)=n-1+q$. In particular, for $n\ge 2$, $trc(T)=strc(T)=2n-3$ if and only if $T$ is a path; and for $n\ge 3$, $trc(T)=strc(T)=n$ if and only if $T$ is a star.
\end{enumerate}
\end{prop}

Our next task is to consider cycles. Recall that $C_n$ denotes the cycle of order $n\ge 3$. The functions $rc(C_n)$ and $src(C_n)$ were determined by Chartrand et al.~\cite{CJMZ}, while $rvc(C_n)$, $srvc(C_n)$ and $trc(C_n)$ were determined by Li and Liu \cite{LL2011}, Lei et al.~\cite{LLLS}, and Liu et al.~\cite{LiuMS}, respectively. We may summerise these results as follows.

\begin{thm}\label{Cnthm}\textup{\cite{CJMZ,LLLS,LL2011,LiuMS}}
\indent\\[-3ex]
\begin{enumerate}
\item[(a)] $rc(C_3)=src(C_3)=1$, and $rc(C_n)=src(C_n)=\lceil\frac{n}{2}\rceil$ for $n\ge 4$.
\item[(b)] For $3\leq n\leq 15$, the values of $rvc(C_n)$ and $srvc(C_n)$ are given in the following table.
\[
\begin{array}{|c||c|c|c|c|c|c|c|c|c|c|c|c|c|c|c|c|}
\hline
n & 3 & 4 & 5 & 6 & 7 & 8 & 9 & 10 & 11 & 12 & 13 & 14 & 15\\
\hline
rvc(C_n) & 0 & 1 & 1 & 2 & 3 & 3 & 3 & 4 & 5 & 5 & 6 & 7 & 7\\
\hline
srvc(C_n) & 0 & 1 & 1 & 2 & 3 & 3 & 3 & 4 & 6 & 5 & 7 & 7 & 8\\
\hline
\end{array}
\]
For $n\geq 16$, we have $rvc(C_n)=srvc(C_n)=\lceil\frac{n}{2}\rceil$.
\item[(c)] For $3\leq n\leq 12$, the values of $trc(C_n)$ are given in the following table.
\[
\begin{array}{|c||c|c|c|c|c|c|c|c|c|c|c|c|c|}
\hline
n & 3 & 4 & 5 & 6 & 7 & 8 & 9 & 10 & 11 & 12\\
\hline
trc(C_n) & 1 & 3 & 3 & 5 & 6 & 7 & 8 & 9 & 11 & 11\\
\hline
\end{array}
\]
For $n\geq 13$, we have $trc(C_n)=n$.
\end{enumerate}
\end{thm}

Note that we have the slightly surprising facts that $rc(C_n)=src(C_n)$, but $rvc(C_n)=srvc(C_n)$ except for $n=11,13,15$; and that $srvc(C_{11})>srvc(C_{12})$. By taking advantage of the fact that $strc(C_n)\geq trc(C_n)$ and the proof of part (c) in \cite{LiuMS}, we have the following result for $strc(C_n)$.

\begin{thm}\label{strc(Cn)}
For $n\ge 3$, we have $strc(C_n)=trc(C_n)$. That is, for $3\le n\le 12$, the values of $strc(C_n)$ are given in the table in Theorem \ref{Cnthm}(c). For $n\geq 13$, we have $strc(C_n)=n$.
\end{thm}

\begin{proof}
One can easily check that $strc(C_3)=1$, $strc(C_4)=3$, and $strc(C_5)=3$. Now, let $n\geq 6$. We need to prove that $strc(C_n)\le trc(C_n)$. Thus by Theorem \ref{Cnthm}(c), we need to prove that $strc(C_n)\leq n-1$ for $6 \leq n\leq 10$ and $n=12$, and $strc(C_n)\leq n$ for $n=11$ and $n\geq 13$. The following facts were shown in the proof of Theorem \ref{Cnthm}(c) in \cite{LiuMS}.
\begin{itemize}
\item For $6\leq n\leq 10$ and $n=12$, there is a total-colouring of $C_n$, using $n-1$ colours, such that every path of length $\lceil\frac{n}{2}\rceil-1$ is total-rainbow, and when $n$ is even, any two opposite vertices of $C_n$ are connected by a total-rainbow path.
\item For $n=11$ and $n\geq 13$, there is a total-colouring of $C_n$, using $n$ colours, such that every path of length $\lceil\frac{n}{2}\rceil$ is total-rainbow.
\end{itemize}
With these total-colourings, it is easy to see that any two vertices $x$ and $y$ of $C_n$ are connected by a total-rainbow $x-y$ path of length at most $\lfloor\frac{n}{2}\rfloor$, which must also be a total-rainbow $x-y$ geodesic. Thus the total-colourings are also strong total rainbow connected colourings, and the upper bound $strc(C_n)\leq trc(C_n)$ follows.
\end{proof}

Next, we consider wheel graphs. The \emph{wheel} $W_n$ of order $n+1\ge 4$ is the graph obtained from the cycle $C_n$ by joining a new vertex $v$ to every vertex of $C_n$. The vertex $v$ is the \emph{centre} of $W_n$. Trivially, we have $rvc(W_3)=srvc(W_3)=0$, and $rvc(W_n)=srvc(W_n)=1$ for $n\ge 4$. The functions $rc(W_n)$ and $src(W_n)$ were determined by Chartrand et al.~\cite{CJMZ}, while $trc(W_n)$ was determined by Liu et al.~\cite{LiuMS}.

\begin{thm}\label{Wnthm}\textup{\cite{CJMZ,LiuMS}}
\indent\\[-3ex]
\begin{enumerate}
\item[(a)] $rc(W_3)=1$, $rc(W_n)=2$ for $4 \le n \le 6$, and $rc(W_n)=3$ for $n \ge 7$.
\item[(b)] $src(W_n)=\lceil\frac{n}{3}\rceil$ for $n\ge 3$.
\item[(c)] $trc(W_3)=1$, $trc(W_n)=3$ for $4\leq n\leq 6$, $trc(W_n)=4$ for $7\leq n\leq 9$, and $trc(W_n)=5$ for $n\geq 10$.
\end{enumerate}
\end{thm}

In the next result, we determine the function $strc(W_n)$. The proof is partially based on the fact that $strc(W_n)\geq trc(W_n)$.

\begin{thm}\label{strc(Wn)}
$strc(W_3)=1$, and $strc(W_n)=\lceil\frac{n}{3}\rceil+1$ for $n\ge 4$.
\end{thm}

\begin{proof}
Let $v$ be the centre of $W_n$, and $v_1,v_2,\ldots,v_n$ be the vertices of $W_n$ in the cycle $C_n$. Since $W_3$ is precisely the complete graph $K_4$, we have $strc(W_3)=1$.

Now, let $n\geq 4$. Since $\diam(W_n)=2$, by Proposition \ref{diam=2} and Theorem \ref{Wnthm}(b), we have $strc(W_n)\leq src(W_n)+1=\lceil \frac{n}{3}\rceil+1$. Also, by Theorem \ref{Wnthm}(c), we have $strc(W_n)\geq trc(W_n)=3=\lceil\frac{n}{3}\rceil+1$ for $4\leq n\leq 6$. It remains to show that $strc(W_n)\geq \lceil\frac{n}{3}\rceil+1$ for $n\geq 7$. Assume the contrary, and suppose that there is a strongly total rainbow connected colouring $c$ of $W_n$, using at most $\lceil\frac{n}{3}\rceil$ colours. Since $n\geq 7$, for each vertex $v_i$, there exists at least one vertex $v_j$ with $j\neq i$ such that the unique $v_i-v_j$ geodesic of length $2$ passes the centre $v$. Thus,  $c(v)\neq c(vv_i)$ for $i=1,2,\ldots,n$. Therefore, the $n$ edges $vv_i$ use at most $\lceil\frac{n}{3}\rceil-1<\frac{n}{3}$ different colours. One can deduce that there exist at least four different edges, say $vv_i$, $vv_j$, $vv_k$, $vv_\ell$, such that $c(vv_i)=c(vv_j)=c(vv_k)=c(vv_\ell)$. Again, since $n\geq 7$, we may assume that the unique $v_i-v_j$ geodesic is precisely the path $v_ivv_j$. So, there is no total-rainbow $v_i-v_j$ geodesic, a contradiction. Consequently, $strc(W_n)\geq \lceil\frac{n}{3}\rceil+1$ for $n\geq 7$.
\end{proof}

Our next aim is to consider complete bipartite graphs $K_{m,n}$. Clearly we have $rc(K_{1,n})=src(K_{1,n})=n$; $rvc(K_{1,1})=srvc(K_{1,1})=0$ and $rvc(K_{m,n})=srvc(K_{m,n})=1$ for $(m,n)\neq (1,1)$; and $trc(K_{1,1})=strc(K_{1,1})=1$ and $trc(K_{1,n})=strc(K_{1,n})=n+1$ for $n\geq 2$. For $2\le m\le n$, the functions $rc(K_{m,n})$ and $src(K_{m,n})$ were determined by Chartrand et al.~\cite{CJMZ}, and the function $trc(K_{m,n})$ was determined by Liu et al.~\cite{LiuMS}.

\begin{thm}\label{Kmnthm}\textup{\cite{CJMZ,LiuMS}}\,
Let $2\leq m\leq n$. We have the following.
\begin{enumerate}
\item[(a)] $rc(K_{m,n})=\min(\lceil\!\sqrt[m]{n}\,\rceil,4)$.
\item[(b)] $src(K_{m,n})=\lceil\!\sqrt[m]{n}\,\rceil$.
\item[(c)] $trc(K_{m,n})=\min(\lceil\!\sqrt[m]{n}\,\rceil+1,7)$.
\end{enumerate}
\end{thm}

In the next result, we will determine $strc(K_{m,n})$ for $2\leq m\leq n$.

\begin{thm}\label{strc(Kmn)}
For $2\leq m\leq n$, we have $strc(K_{m,n})=\lceil\!\sqrt[m]{n}\,\rceil+1$.
\end{thm}

\begin{proof}
Since $\diam(K_{m,n})=2$, we have $strc(K_{m,n})\leq src(K_{m,n})+1=\lceil\!\sqrt[m]{n}\,\rceil+1$ by Proposition \ref{diam=2} and Theorem \ref{Kmnthm}(b).

Now we prove the lower bound $strc(K_{m,n})\geq \lceil\!\sqrt[m]{n}\,\rceil+1$. This proof will be a slight modification of the proof of the lower bound of Theorem \ref{Kmnthm}(c) in \cite{LiuMS}, but we provide it for the sake of clarity. Let the classes of $K_{m,n}$ be $U=\{u_1,\dots,u_m\}$ and $V$, where $|V|=n$. Let $b=\lceil\!\sqrt[m]{n}\,\rceil\geq 2$. If $m\leq n\leq 2^m$, then $strc(K_{m,n})\geq 3=b+1$. Now let $n>2^m$, so that $b\geq 3$. We have $(b-1)^m<n\le b^m$. Let $c$ be a total-colouring of $K_{m,n}$, using colours from $\{1,\dots , b\}$. For $v\in V$, assign $v$ with the vector $\vec{v}$ of length $m$, where $\vec{v}_i=c(u_iv)$ for $1\le i\le m$. For two partitions $\mathcal P$ and $\mathcal P'$ of $V$, we say that $\mathcal P$ \emph{refines} $\mathcal P'$, written $\mathcal P'\prec\mathcal P$, if for all $A\in\mathcal P$, we have $A\subseteq B$ for some $B\in\mathcal P'$. In other words, $\mathcal P$ can be obtained from $\mathcal P'$ by partitioning some of the sets of $\mathcal P'$. We define a sequence of refining partitions $\mathcal P_0\prec\mathcal P_1\prec\cdots\prec\mathcal P_m$ of $V$, with $|\mathcal P_i|\le (b-1)^i$ for $0\le i\le m$, as follows. Initially, set $\mathcal P_0=\{V\}$. Now, for $1\le i\le m$, suppose that we have defined $\mathcal P_{i-1}$ with $|\mathcal P_{i-1}|\le (b-1)^{i-1}$. Let $\mathcal P_{i-1}=\{A_1,\dots , A_\ell\}$, where $\ell\le (b-1)^{i-1}$. Define $\mathcal P_i$ as follows. For $1\le q\le\ell$ and $A_q\in \mathcal P_{i-1}$, let
\begin{align}
B_1^q &= \{v\in A_q : \vec{v}_i=c(u_i)\textup{ or }c(u_i)+1\textup{ (mod }b)\},\nonumber\\
B_r^q &= \{v\in A_q : \vec{v}_i=c(u_i)+r\textup{ (mod }b)\}\textup{, for }2\le r\le b-1.\nonumber
\end{align}

Let $\mathcal P_i=\{B_r^q:1\le q\le\ell$, $1\le r\le b-1$ and $B_r^q\neq\emptyset\}$, so that $\mathcal P_i$ is a partition of $V$ with $|\mathcal P_i|\le (b-1)^i$ and $\mathcal P_{i-1}\prec\mathcal P_i$. Proceeding inductively, we obtain the partitions $\mathcal P_0\prec\mathcal P_1\prec\cdots\prec\mathcal P_m$ of $V$, with $|\mathcal P_i|\le (b-1)^i$ for $0\le i\le m$. Now, observe that for every $1\le i\le m$, and any two vertices $y$ and $z$ in the same set in $\mathcal P_i$, the path $yu_iz$ is not total-rainbow, since $c(u_iy)=\vec{y}_i$ and $c(u_iz)=\vec{z}_i$ are either in $\{c(u_i),c(u_i)+1\}$ (mod $b)$, or they are both $c(u_i)+r$ (mod $b)$ for some $2\le r\le b-1$. Since $n>(b-1)^m\ge |\mathcal P_m|$, there exists a set in $\mathcal P_m$ with at least two vertices $w$ and $x$, and since $\mathcal P_1\prec\cdots\prec\mathcal P_m$, this means that $w$ and $x$ are in the same set in $\mathcal P_i$ for every $1\le i\le m$. Therefore, $wu_ix$ is not a total-rainbow path for every $1\le i\le m$. Since the paths $wu_ix$ are all the possible $w-x$ geodesics (with length $2$) in $K_{m,n}$, it follows that there does not exist a total-rainbow $w-x$ geodesic. Hence, $c$ is not a strongly total rainbow connected colouring of $K_{m,n}$, and $strc(K_{m,n})\ge b+1$.
\end{proof}

To conclude this section, we consider complete multipartite graphs. Let $K_{n_1,\ldots,n_t}$ denote the complete multipartite graph with $t\ge 3$ classes, where $1\le n_1\le\cdots\le n_t$ are the class sizes. Clearly, we have $rvc(K_{n_1,\ldots,n_t})=srvc(K_{n_1,\ldots,n_t})=0$ (resp.~$1$) if $n_t=1$ (resp.~$n_t\ge 2$). The functions $rc(K_{n_1,\ldots,n_t})$ and $src(K_{n_1,\ldots,n_t})$ were determined by Chartrand et al.~\cite{CJMZ}, and the function $trc(K_{n_1,\ldots,n_t})$ was determined by Liu et al.~\cite{LiuMS}, as follows.

\begin{thm}\label{mpthm}\textup{\cite{CJMZ,LiuMS}}\,
Let $G=K_{n_1,\dots,n_t}$, where $t\ge 3$, $1\le n_1\le\cdots\le n_t$, $m=\sum_{i=1}^{t-1}n_i$ and $n_t=n$. Then, the functions $rc(G), src(G)$ and $trc(G)$ are given in the following table.
\[
\begin{array}{|c||c|c|c|}
\hline
& n=1 & n\ge 2\textup{ \emph{and} }m>n & m\le n\\
\hline
rc(G) & 1 & 2 & \min(\lceil\!\sqrt[m]{n}\,\rceil,3)\\
\hline
src(G) & 1 & 2 & \lceil\!\sqrt[m]{n}\,\rceil\\
\hline
trc(G) & 1 & 3 & \min(\lceil\!\sqrt[m]{n}\,\rceil+1,5)\\
\hline
\end{array}
\]
%
%
\end{thm}

Here, we determine the function $strc(K_{n_1,\ldots,n_t})$ for $t\ge 3$.

\begin{thm}\label{strc(Kn1nt)}
Let $t\ge 3$, $1\le n_1\le\cdots\le n_t$, $m=\sum_{i=1}^{t-1}n_i$ and $n_t=n$. Then,
\[
strc(K_{n_1,\dots,n_t})=
\left\{
\begin{array}{l@{\quad\quad}l}
1 & \textup{\emph{if} }n=1\textup{\emph{,}}\\
3 & \textup{\emph{if} }n\ge 2\textup{\emph{ and }}m>n\textup{\emph{,}}\\
\lceil\!\sqrt[m]{n}\,\rceil+1 & \textup{\emph{if} }m\le n.
\end{array}
\right.
\]
\end{thm}

\begin{proof}
Write $G$ for $K_{n_1,\dots,n_t}$, and let $V_i$ be the $i$th class (with $n_i$ vertices) for $1\le i\le t$. If $n=1$, then $G=K_t$ and $strc(G)=1$. Now for $n\ge 2$, we have $strc(G)\ge 3$. For the case $n\ge 2$ and $m>n$, we have  $src(G)=2$ by Theorem \ref{mpthm}. Since $\diam(G)=2$, by Proposition \ref{diam=2}, we have $strc(G)\le src(G)+1=3$. Thus, $strc(G)=3$.

Now, let $m\le n$. For this case, we have $src(G)= \lceil\!\sqrt[m]{n}\,\rceil$ by Theorem \ref{mpthm}. Again by Proposition \ref{diam=2}, we have the upper bound $strc(G)\le src(G)+1=\lceil\!\sqrt[m]{n}\,\rceil+1$. It remains to prove the lower bound $strc(G)\ge\lceil\!\sqrt[m]{n}\,\rceil+1$. Let $b=\lceil\!\sqrt[m]{n}\,\rceil\ge 2$. If $m\le n\le 2^m$, then $strc(G)\ge 3=b+1$. Now let $n>2^m$, so that $b\ge 3$. We have $(b-1)^m<n\le b^m$. Suppose that we have a total-colouring $c$ of $G$, using at most $b$ colours. Note that $K_{m,n}$ is a spanning subgraph of $G$ with classes $U=V_1\cup\cdots\cup V_{t-1}$ and $V_t$. We can restrict the total-colouring $c$ to $K_{m,n}$ and apply the same argument involving the refining partitions as in Theorem \ref{strc(Kmn)}. We have vertices $w,x\in V_t$ such that all of the paths $wux$, for $u\in U$, are not total-rainbow. Since these paths are all the possible $w-x$ geodesics in $G$ (of length $2$), it follows that there does not exist a total-rainbow $w-x$ geodesic in $G$. Therefore, $c$ is not a strongly total rainbow connected colouring of $G$, and $strc(G)\ge b+1$.
\end{proof}

\section{Comparing the rainbow connection numbers}\label{comparesect}

Our aim in this section is to compare the various rainbow connection parameters. In \cite{KY}, Krivelevich and Yuster observed that for $rc(G)$ and $rvc(G)$, we cannot generally find an upper bound for one of the parameters in terms of the other. Indeed, let $s\ge 2$. By taking $G=K_{1,s}$, we have $rc(G) = s$ and $rvc(G) = 1$. On the other hand, let the graph $G_s$ be constructed as follows. Take $s$ vertex-disjoint triangles and, by designating a vertex from each triangle, add a complete graph $K_s$ on the designated vertices. Then $rc(G_s) \le 4$ and $rvc(G_s) = s$.

We may consider the analogous situation for the parameters $src(G)$ and $srvc(G)$. Again by taking $G=K_{1,s}$, we see that $src(G)=s$ and $srvc(G)=1$, so that $src(G)$ can be arbitrarily larger than $srvc(G)$. Rather surprisingly, unlike the situation for the functions $rvc(G)$ and $rc(G)$, we are uncertain if $srvc(G)$ can also be arbitrarily larger than $src(G)$. We propose the following problem.

\begin{prob}\label{srvc>srcprob}
Does there exist an infinite family of connected graphs $\mathcal F$ such that, $src(G)$ is bounded on $\mathcal F$, while $srvc(G)$ is unbounded?
\end{prob}

When considering the total rainbow connection number in addition, we have the following trivial inequalities.
\begin{align}
trc(G) & \ge\max(rc(G),rvc(G)),\label{maxeq1}\\
strc(G) & \ge\max(src(G),srvc(G)).\label{maxeq2}
\end{align}

In \cite{LiuMS}, Liu et al.~considered how close and how far apart the terms in the inequality (\ref{maxeq1}) can be. They observed that by considering Krivelevich and Yuster's construction as described above, we have $trc(G_s)=rvc(G_s)=s$ for $s\ge 13$. Also, as mentioned in the remark after the proof of Theorem \ref{thm1}, if $G=(C_5)_{u\to K}$ is a graph obtained by expanding a vertex $u$ of the cycle $C_5$ into a clique $K$, then we have $trc(G)=rc(G)=3$. Thus, $trc(G)$ can be equal to each of $rvc(G)$ and $rc(G)$ for infinitely many graphs $G$. On the other hand, Liu et al.~also remarked that, given $1 \le t < s$, there exists a graph $G$ such that $trc(G) \ge s$ and $rvc(G) = t$. Indeed, we can let $G=B_{s,t}$ be the graph obtained by taking the star $K_{1,s}$ and identifying the centre with one end-vertex of the path of length $t$ (this graph $B_{s,t}$ is a \emph{broom}). Also, for $s\ge 13$, we can again consider the graphs $G_s$ and obtain $trc(G_s)=s$ and $rc(G_s)\le 4$. Thus, $trc(G)$ can also be arbitrarily larger than each of $rvc(G)$ and $rc(G)$. For the difference between the terms $trc(G)$ and $\max(rc(G),rvc(G))$, one can consider $G$ to be the path of length $s$, and obtain $trc(G)=2s-1$ and $\max(rc(G),rvc(G))=s$, so that $trc(G)-\max(rc(G),rvc(G))=s-1$ can be arbitrarily large. However, for this simple example, the term $\max(rc(G),rvc(G))$ is unbounded in $s$. In the final problem in \cite{LiuMS}, Liu et al.~asked the question of whether there exists an infinite family of connected graphs $\mathcal F$ such that, $\max(rc(G),rvc(G))$ is bounded on $\mathcal F$, while $trc(G)$ is unbounded. This open problem appears to be much more challenging.

Here, we consider the analogous situations for the terms in the inequality (\ref{maxeq2}). From the previous remarks and results, we can easily obtain the following.

\begin{thm}\label{comparisonlthm}
\indent\\[-3ex]
\begin{enumerate}
\item[(a)] There exist infinitely many graphs $G$ with $strc(G)=src(G)=3$.
\item[(b)] Given $s\ge 13$, there exists a graph $G$ with $strc(G)=srvc(G)=s$.
\item[(c)] Given $1 \le t < s$, there exists a graph $G$ such that $strc(G) \ge s$ and $srvc(G) = t$.
\end{enumerate}
\end{thm}

\begin{proof}
(a) Let $G=(C_5)_{u\to K}$ as described earlier. We have $trc(G)=rc(G)=3$. By Theorem \ref{thm1}(c), we have $strc(G)=3$. Therefore by (\ref{ineq1}) and (\ref{maxeq2}), we have $3=strc(G)\ge src(G)\ge rc(G)=3$, so that $strc(G)=src(G)=3$.\\[1ex]
\indent (b) We use the following construction which was given by Lei et al.~\cite{LLMS}. For $s\ge 13$, let $H_s$ be the graph as follows. First, we take the graph $G_s$ from before, where $u_1,\dots, u_s$ are the vertices of the $K_s$, and the remaining vertices are $v_i, w_i$, where $u_iv_iw_i$ is a triangle,  for $1\le i\le s$. We then add new vertices $z_1,\dots,z_s$, and connect the edges $u_iz_i, u_{i+1}z_i, v_iz_i, w_iz_{i+4}$, for $1\le i\le s$, where all indices are taken modulo $s$. In \cite{LLMS}, Lei et al.~proved that $strc(H_s)=srvc(H_s)=s$.\\[1ex]
\indent(c) Since the broom $G=B_{s,t}$ as described earlier is a tree, it is clear that $strc(G)=trc(G)\ge s$ and $srvc(G) =rvc(G)= t$.
\end{proof}

As before, if $G$ is the path of length $s$, then we have $strc(G)-src(G)=strc(G)-\max(src(G),srvc(G))=s-1$, so that the two  differences can both be arbitrarily large. But the terms $src(G)$ and $\max(src(G),srvc(G))$ are unbounded in $s$. Similar to the question of Liu et al.~in \cite{LiuMS} and Problem \ref{srvc>srcprob}, we may ask the following question.

\begin{prob}
Does there exist an infinite family of connected graphs $\mathcal F$ such that, $src(G)$ is bounded on $\mathcal F$, while $strc(G)$ is unbounded? Similarly, does there exist an infinite family of connected graphs $\mathcal F$ such that, $\max(src(G),srvc(G))$ is bounded on $\mathcal F$, while $strc(G)$ is unbounded?
\end{prob}

Now, we proceed to the final part of this section. Recall that the following inequalities hold for a connected graph $G$.
\[
rc(G)\le src(G),\quad rvc(G)\le srvc(G),\quad\textup{and}\quad trc(G)\le strc(G).
\]
Chartrand et al.~\cite{CJMZ} considered the following question: \emph{Given positive integers $a\le b$, does there exist a graph $G$ such that $rc(G)=a$ and $src(G)=b$?} They gave positive answers for $a=b$, and $3\le a<b$ with $b\ge\frac{5a-6}{3}$. Chern and Li \cite{CL} then improved this result as follows.

\begin{thm}\label{CLthm}\textup{\cite{CL}}\,
Let $a$ and $b$ be positive integers. Then there exists a connected graph $G$ such that $rc(G) = a$ and $src(G) = b$ if and only if $a = b \in \{1, 2\}$ or $3 \le a \le b$.
\end{thm}

Theorem \ref{CLthm} was an open problem of Chartrand et al., and it completely characterises all possible pairs $a$ and $b$ for the above question. Subsequently, Li et al.~\cite{LMS} studied the rainbow vertex-connection analogue, and they proved the following result.

\begin{thm}\label{LiMaoShithm}\textup{\cite{LMS}}\,
Let $a$ and $b$ be integers with $a \ge 5$ and $b \ge \frac{7a - 8}{5}$. Then there exists a connected graph $G$ such that $rvc(G) = a$ and $srvc(G) = b$.
\end{thm}

Here, we will improve Theorem \ref{LiMaoShithm}, and also study the total rainbow connection version of the problem. We will prove Theorems \ref{rvcpresthm} and \ref{trcpresthm} below, where we will completely characterise all pairs of positive integers $a$ and $b$ such that, there exists a graph $G$ with $rvc(G) = a$ and $srvc(G) = b$ (resp.~$trc(G) = a$ and $strc(G) = b$).


\begin{thm}\label{rvcpresthm}
Let $a$ and $b$ be positive integers. Then there exists a connected graph $G$ such that $rvc(G) = a$ and $srvc(G) = b$ if and only if $a = b \in \{1,2\}$ or $3 \le a \le b$.
\end{thm}

\begin{thm}\label{trcpresthm}
Let $a$ and $b$ be positive integers. Then there exists a connected graph $G$ such that $trc(G) = a$ and $strc(G) = b$ if and only if $a = b \in \{1,3,4\}$ or $5 \le a \le b$.
\end{thm}

To prove Theorems \ref{rvcpresthm} and \ref{trcpresthm}, we first prove three auxiliary lemmas.

\begin{lem}\label{rvcpreslem1}
For every $b\ge 3$, there exists a graph $G$ such that $rvc(G)=3$ and $srvc(G)=b$.
\end{lem}

\begin{proof}
We construct a graph $F_b$ as follows. We take a complete graph $K_{2b}$, say with vertices $u_1,\dots,u_{2b}$, and further vertices $v_1,\dots,v_{2b},w_1,\dots,w_{2b}$. For $1\le i\le 2b$, we connect the edges $u_iv_i,u_iv_{i-1},u_iw_i,w_iv_i,w_iv_{i-1}$. Throughout, the indices of the vertices $u_i,v_i,w_i$ are taken modulo $2b$. We show that $rvc(F_b)=3$ and $srvc(F_b)=b$.

Suppose firstly that we have a vertex-colouring of $F_b$, using at most two colours. Since $2b\ge 6$, we may assume that $u_1$ and $u_\ell$ have the same colour, for some $3\le\ell\le 2b-1$. Then note that $w_1u_1u_\ell w_\ell$ is the unique $w_1-w_\ell$ geodesic, with length $3$. Thus, there does not exist a vertex-rainbow $w_1-w_\ell$ path, and we have $rvc(F_b)\ge 3$. Now, we define a vertex-colouring $f$ of $F_b$ as follows. Let $f(u_i)=1$ if $i$ is odd, and $f(u_i)=2$ if $i$ is even. Let $f(z)=3$ for all other vertices $z$. It is easy to check that $f$ is a rainbow vertex-connected colouring for $F_b$. For example, to connect $w_1$ to $w_i$ with a vertex-rainbow path, where $3\le i\le 2b-1$, we may take $w_1u_1u_iw_i$ if $i$ is even, and $w_1u_1u_{i-1}v_{i-1}w_i$ if $i$ is odd. Thus $rvc(F_b)\le 3$, and we have $rvc(F_b)=3$.

Next, suppose that we have a vertex-colouring of $F_b$, using fewer than $b$ colours. Then, three of the vertices $u_i$ have the same colour, so we may assume that $u_1$ and $u_\ell$ have the same colour, for some $3\le\ell\le 2b-1$. Note that $w_1u_1u_\ell w_\ell$ is the unique $w_1-w_\ell$ geodesic (with length $3$). Thus, there does not exist a vertex-rainbow $w_1-w_\ell$ geodesic, and we have $srvc(F_b)\ge b$. Now, we define a vertex-colouring $g$ of $F_b$ as follows. Let $g(u_i)=\lceil\frac{i}{2}\rceil$ for $1\le i\le 2b$, and $g(z)=1$ for all other vertices $z$. We show that $g$ is a strongly rainbow vertex-connected colouring for $F_b$. It is easy to see that each vertex $u_i$ is at distance at most $2$ from every other vertex. Thus, it suffices to check that $v_1$ is connected to each $v_i$ and $w_j$ by a vertex-rainbow geodesic, and similarly for $w_1$ to each $w_j$. Now, $d(v_1,w_1)=d(v_1,w_2)=1$ and $d(v_1,v_2)=d(v_1,v_{2b})=2$. Also, $d(v_1,v_i)=d(v_1,w_j)=3$ for $3\le i\le 2b-1$ and $3\le j\le 2b$, whence $v_1u_1u_iv_i$ and $v_1u_1u_jw_j$ are vertex-rainbow $v_1-v_i$ and $v_1-w_j$ geodesics. Likewise, $d(w_1,w_2)=d(w_1,w_{2b})=2$, and $d(w_1,w_j)=3$ for $3\le j\le 2b-1$, whence $w_1u_1u_jw_j$ is a vertex-rainbow $w_1-w_j$ geodesic. Thus $srvc(F_b)\le b$, and we have $srvc(F_b)=b$.
\end{proof}

\begin{lem}\label{rvcpreslem2}
For every $4\le a\le b$, there exists a connected graph $G$ such that $rvc(G)=a$ and $srvc(G)=b$.
\end{lem}

\begin{proof}
We construct a graph $F_{a,b}$ as follows. Let $n=2(b-1)(b-a+2)\ge 12$. We take a set of vertices $V=\{v_1,\dots, v_n\}$, and another vertex $u$ and a path $u_0\cdots u_{a-3}$. We add the paths $uw_iv_i$ and $u_{a-3}x_iv_i$ for $1\le i\le n$, and then the edges $v_\ell v_{\ell+1},w_\ell w_{\ell+1},x_\ell x_{\ell+1}$ for $1\le\ell\le n$ with $\ell$ odd. Let $U=\{u_0,\dots,u_{a-3}\}$, $W=\{w_1,\dots,w_n\}$ and $X=\{x_1,\dots,x_n\}$. Note that we have perfect matchings within the sets $V, W$ and $X$. We show that $rvc(F_{a,b})=a$ and $srvc(F_{a,b})=b$.

Clearly we have $rvc(F_{a,b})\ge\diam(F_{a,b})-1=a$. Now, we define a vertex-colouring $c$ of $F_{a,b}$ as follows. Let $c(u_j)=j$ for $1\le j\le a-3$. For $1\le i\le n$, let $c(w_{i+1})=c(x_i)=a-2$ if $i$ is odd, and $c(w_{i-1})=c(x_i)=a-1$ if $i$ is even. Let $c(z)=a$ for all other vertices $z$. It is easy to check that $c$ is a rainbow vertex-connected colouring for $F_{a,b}$. For example, for $i\neq 2$, to connect $v_1$ to $v_i$ with a vertex-rainbow path, we may take $v_1x_1u_{a-3}x_iv_i$ if $i$ is even, and $v_1x_1u_{a-3}x_{i+1}v_{i+1}v_i$ if $i$ is odd, since $a\ge 4$. Thus $rvc(F_{a,b})\le a$, and we have $rvc(F_{a,b})=a$.

Next, suppose that there exists a strongly rainbow vertex-connected colouring $f$ of $F_{a,b}$, using at most $b-1$ colours, say colours $1,2,\dots,b-1$. Then note that for every $1\le i\le n$, the unique $u_0-v_i$ geodesic is $u_0u_1\cdots u_{a-3}x_iv_i$. Thus we may assume that $f(u_j)=j$ for $1\le j\le a-3$, so that $f(x_i)\in\{a-2,a-1,\dots,b-1\}$ for $1\le i\le n$. Also, we have $f(w_i), f(u)\in\{1,\dots,b-1\}$ for $1\le i\le n$. For $a-2\le p\le b-1$ and $1\le q\le b-2$, we define the set $A_{p,q}\subset V$ where
\begin{align}
A_{p,1} &= \{v_i\in V : f(x_i)=p\textup{ and }f(w_i)=f(u)\textup{ or }f(u)+1\textup{ (mod }b-1)\},\nonumber\\
A_{p,q} &= \{v_i\in V : f(x_i)=p\textup{ and }f(w_i)=f(u)+q\textup{ (mod }b-1)\}\textup{, for $q\ge 2$.}\nonumber
\end{align}
Note that $V=\bigcup_{p,q}\{A_{p,q}:A_{p,q}\neq\emptyset\}$ is a partition of $V$ with at most $(b-2)(b-a+2)$ parts. Since $n=2(b-1)(b-a+2)$, there exists a set $A_{r,s}$ with at least three vertices. Thus, we may assume that $v_1,v_\ell\in A_{r,s}$ with $\ell\neq 2$. Observe that the path $v_1x_1u_{a-3}x_\ell v_\ell$ is not vertex-rainbow, since $f(x_1)=f(x_\ell)=r$. Also, the path $v_1w_1uw_\ell v_\ell$ is not vertex-rainbow, since $f(w_1)$ and $f(w_\ell)$ are either in $\{f(u),f(u)+1\}$ (mod $b-1)$, or they are both $f(u)+s$ (mod $b-1)$. Since these two paths are the only $v_1-v_\ell$ geodesics (with length $4$), we have a contradiction. Thus, $srvc(F_{a,b})\ge b$.

Finally, we define a vertex-colouring $g$ of $F_{a,b}$, using colours $1,2,\dots,b$, as follows. Let $g(u_j)=j$ for $1\le j\le a-3$, and $g(u)=g(u_0)=g(v_i)=b$ for $1\le i\le n$. Now, note that there are $(b-1)(b-a+2)$ pairs $\{v_\ell, v_{\ell+1}\}$ with $\ell$ odd, and also $(b-1)(b-a+2)$ distinct vectors of length $2$, whose first coordinate is in $\{a-2,\dots,b-1\}$ and second coordinate is in $\{1,\dots,b-1\}$. Thus we may assign these distinct vectors to all vertices of $V$ such that, both vertices of a pair $\{v_\ell, v_{\ell+1}\}$ with $\ell$ odd receive the same vector (so that every vector appears exactly twice). If $v_\ell$ and $v_{\ell+1}$ have been assigned with the vector $\vec{v}$, then we set $g(x_\ell)=g(x_{\ell+1})=\vec{v}_1\in \{a-2,\dots,b-1\}$, and $g(w_\ell)=g(w_{\ell+1})=\vec{v}_2\in\{1,\dots,b-1\}$. We show that $g$ is a strongly rainbow vertex-connected colouring for $F_{a,b}$. We must show that for every $x,y\in V(F_{a,b})$, there is a vertex-rainbow $x-y$ geodesic.

\begin{itemize}
\item If $x\in U$ and $y\neq u$, then it is easy to find a vertex-rainbow $x-y$ geodesic. For example, if $x=u_j$ and $y=w_i$, then we take $u_j\cdots u_{a-3}x_iv_iw_i$. If $x=u_j$ and $y=u$, then we take $u_j\cdots u_{a-3}x_\ell v_\ell w_\ell u$, where $v_\ell$ is assigned with the vector $(a-2,b-1)$. Similarly, it is easy to deal with the case when $x=u$ and $y\in V\cup W\cup X$.
\item Now we consider the case $x,y\in V\cup W\cup X$. Firstly, the cases $x,y\in W$ and $x,y\in X$ are clear, since $d(x,y)\le 2$. Next, suppose that $x\in V$, say $x=v_1$. Then the case $y\in\{w_1,x_1,v_2,w_2,x_2\}$ is clear, since we have $d(x,y)\le 2$. If $y=w_\ell$ (resp.~$x_\ell$) for some $\ell\neq 2$, then we take $v_1w_1uw_\ell$ (resp.~$v_1x_1u_{a-3}x_\ell$). If $y=v_\ell$ for some $\ell\neq 2$, then $x$ and $y$ are assigned with different vectors, say $\vec{x}\neq\vec{y}$. If $\vec{x}_1\neq\vec{y}_1$, then we take $v_1x_1u_{a-3}x_\ell v_\ell$, and if $\vec{x}_2\neq\vec{y}_2$, then we take $v_1w_1uw_\ell v_\ell$. Finally, it remains to consider the case $x\in W$ and $y\in X$. We may assume that $x=w_1$ and $y=x_\ell$ for some $1\le \ell\le n$. We take $w_1v_1x_1$ if $\ell=1$; $w_1v_1x_1x_2$ if $\ell=2$; and $w_1v_1x_1u_{a-3}x_\ell$ if $\ell\ge 3$.
\end{itemize}

We always have a vertex-rainbow $x-y$ geodesic, so that $g$ is a strongly rainbow vertex-connected colouring. Therefore $srvc(F_{a,b})\le b$, and we have $srvc(F_{a,b})=b$.
\end{proof}

\begin{lem}\label{trcpreslem}
For every $5\le a<b$, there exists a connected graph $G$ such that $trc(G)=a$ and $strc(G)=b$.
\end{lem}

\begin{proof}
We consider the complete multipartite graph $K_{1,\dots,1,n}$, where there are $m\ge 2$ singleton classes, say $\{u_1\},\dots,\{u_m\}$. Let $U=\{u_1,\dots,u_m\}$, and $V$ be the class with $n$ vertices. Given $5\le a<b$, let $G_{a,b,m}$ be the graph constructed as follows. We take $K_{1,\dots,1,n}$, and set $n=(b-2)^m+1$. We then add $a-1\ge 4$ pendent edges at $u_1$, say $W=\{w_1,\dots,w_{a-1}\}$ is the set of pendent vertices. We claim that for sufficiently large $m$, we have $trc(G_{a,b,m})=a$ and $strc(G_{a,b,m})=b$.

Since the bridges of $G_{a,b,m}$ are the $a-1$ pendent edges, and the only cut-vertex is $u_1$, clearly we have $trc(G_{a,b,m})\ge a$ by Proposition \ref{bridges}. Now we define a total-colouring $f$ of $G_{a,b,m}$ as follows. Let $f(u_1w_\ell)=\ell$ for $1\le \ell\le a-1$. For every $v\in V$, let $f(u_1v)=1$, and $f(u_iv)=2$ for all $2\le i\le m$. Let $f(u_iu_j)=4$ for all $1\le i<j\le m$. Let $f(u_1)=a$, and $f(z)=3$ for all $z\in V(G_{a,b,m})\setminus \{u_1\}$. We claim that $f$ is a total rainbow connected colouring for $G_{a,b,m}$. We need to show that for every $x,y\in V(G_{a,b,m})$, there is a total-rainbow $x-y$ path. Since $u_1$ is connected to all other vertices, it suffices to consider $x,y\in V(G_{a,b,m})\setminus\{u_1\}$. If $x,y\not\in W$ and $x,y$ are not adjacent, then $x,y\in V$, in which case we take the path $xu_1u_2y$. Now suppose $x\in W$. Then we can take the path $xu_1y$, unless if $x=w_1$ and $y\in V$, in which case we take $xu_1u_2y$; or $x=w_4$ and $y\in U\setminus\{u_1\}$, in which case we take $xu_1vy$ for some $v\in V$.
Thus $f$ is a total rainbow connected colouring for $G_{a,b,m}$, and $trc(G_{a,b,m})\le a$. We have $trc(G_{a,b,m})=a$.

Now, suppose that we have a total-colouring of $G_{a,b,m}$, using fewer than $b$ colours. Note that $\lceil\!\sqrt[m]{n}\,\rceil+1=b$, so that by Theorem \ref{strc(Kn1nt)}, for the copy of $K_{1,\dots,1,n}$, we have $strc(K_{1,\dots,1,n})=b$. It follows that when restricted to the $K_{1,\dots,1,n}$, there are two vertices $w,x$ that are not connected by a total-rainbow $w-x$ geodesic. This means that we have $w,x\in V$, and the paths $xuw$, for $u\in U$, are all not total-rainbow. Since these paths are also all the possible $w-x$ geodesics in $G_{a,b,m}$, we do not have a total-rainbow $w-x$ geodesic in $G_{a,b,m}$. Thus $strc(G_{a,b,m})\ge b$.

It remains to prove that $strc(G_{a,b,m})\le b$. Let $m$ be sufficiently large so that $(b-1)^{m-1}>(b-2)^m$. This inequality holds if $m>\frac{\log(b-1)}{\log(b-1)-\log(b-2)}$. Thus, we have $(b-1)^{m-1}\ge n$. We define a total-colouring $g$ of $G_{a,b,m}$ as follows. Let $g(u_1w_\ell)=\ell$ for $1\le \ell\le a-1$. Let $g(u_1)=a$, and $g(u_1v)=g(u_iu_j)=g(z)=b$ for all $v\in V$, $1\le i<j\le m$, and $z\in V(G_{a,b,m})\setminus\{u_1\}$. Now since $(b-1)^{m-1}\ge n$, we may assign distinct vectors of length $m-1$ to the vertices of $V$, with entries from $\{1,2,\dots,b-1\}$. Suppose that $v\in V$ has been assigned with the vector $\vec{v}$. We let $g(u_{i+1}v)=\vec{v}_i$ for $1\le i\le m-1$ and $v\in V$. We claim that $g$ is a strongly total rainbow connected colouring for $G_{a,b,m}$. Similar to before, it suffices to show that for all $x,y\in V(G_{a,b,m})\setminus\{u_1\}$, there is a total-rainbow $x-y$ geodesic. If $x,y\not\in W$ and $x,y$ are not adjacent, then $x,y\in V$. We have $\vec{x}_i\neq\vec{y}_i$ for some $1\le i\le m-1$, so that we can take the geodesic $xu_{i+1}y$. If $x\in W$, then we can take the geodesic $xu_1y$. Thus $g$ is a strongly total rainbow connected colouring for $G_{a,b,m}$, and $strc(G_{a,b,m})\le b$. We have $strc(G_{a,b,m})=b$.
\end{proof}

%
%

We can now prove Theorems \ref{rvcpresthm} and \ref{trcpresthm}.

\begin{proof}[Proof of Theorem \ref{rvcpresthm}]
Suppose that there exists a connected graph $G$ such that $rvc(G) = a$ and $srvc(G) = b$. Then obviously we have $a\le b$. If $a=1$ (resp.~$a=2$), then Theorem \ref{thm1}(c)(ii) (resp.~(c)(iii)) gives $b=1$ (resp.~$b=2$). Therefore, we have either $a = b \in \{1,2\}$, or $3 \le a \le b$.

Conversely, given $a,b$ such that either $a=b\in\{1,2\}$ or $3 \le a \le b$, we show that there exists a connected graph $G$ with $rvc(G) = a$ and $srvc(G) = b$. Obviously if $a=b\ge 1$, then $rvc(G)=srvc(G)=a$ if $G$ is the path of length $a+1$. The remaining cases satisfy $3\le a\le b$, and these are covered by Lemmas \ref{rvcpreslem1} and \ref{rvcpreslem2}. Thus Theorem \ref{rvcpresthm} follows.
\end{proof}

\begin{proof}[Proof of Theorem \ref{trcpresthm}]
Suppose that there exists a connected graph $G$ such that $trc(G) = a$ and $strc(G) = b$. Then obviously we have $a\le b$. If $a=1$ (resp.~$a=3$, $a=4$), then Theorem \ref{thm1}(a) (resp.~(c)(iv), (c)(v)) gives $b=1$ (resp.~$b=3$, $b=4$). Theorem \ref{thm1}(a) and (b) also imply that $a,b\neq 2$. Therefore, we have either $a = b \in \{1,3,4\}$, or $5 \le a \le b$.

Conversely, given $a,b$ such that either $a=b\in\{1,3,4\}$ or $5 \le a \le b$, we show that there is a connected graph $G$ with $trc(G) = a$ and $strc(G) = b$. Obviously, if $a=b=1$, then $trc(G)=strc(G)=1$ if $G$ is any non-trivial complete graph, and if $a=b\ge 3$, then $trc(G)=strc(G)=a$ if $G$ is the star of order $a$. The remaining cases satisfy $5\le a<b$, and these are covered by Lemma \ref{trcpreslem}. Thus Theorem \ref{trcpresthm} follows.
\end{proof}

%

\section*{Acknowledgements}
Lin Chen, Xueliang Li and Jinfeng Liu are supported by the National Science Foundation of China (Nos.~11371205 and 11531011). Henry Liu is supported by China Postdoctoral Science Foundation (Nos.~2015M580695, 2016T90756), and International Interchange Plan of CSU. Henry Liu would also like to thank the Chern Institute of Mathematics, Nankai University, for their generous hospitality. He was able to carry out part of this research during his visit there.

\end{document}